\documentclass[a4paper,12pt]{amsart}
\usepackage{amssymb,amsmath,amsthm,mathrsfs,enumerate,graphicx,color}
\usepackage[pdfpagelabels,colorlinks,linkcolor=MidnightBlue,citecolor=OliveGreen,urlcolor=green]{hyperref}
\usepackage[dvipsnames]{xcolor}
\usepackage[fontsize=12pt]{scrextend}

\usepackage{esint}
\usepackage{stix}
\newtheorem{thm}{Theorem}[section]

\newtheorem{lem}[thm]{Lemma}
\newtheorem{prop}[thm]{Proposition}
\newtheorem{defn}[thm]{Definition}
\newtheorem{rem}[thm]{Remark}
\renewenvironment{proof}{{\bfseries Proof.}}{\hfill\rule{0.5em}{0.5em}}
\newenvironment{pot}[1]{{\bfseries Proof of Theorem {\ref{#1}}.}}{\hfill\rule{0.5em}{0.5em}}

\usepackage{scalerel}
\usepackage[usestackEOL]{stackengine}
\def\dashint{\,\ThisStyle{\ensurestackMath{\stackinset{c}{.2\LMpt}{c}{.5\LMpt}{\SavedStyle-}{\SavedStyle\phantom{\int}}}\setbox0=\hbox{$\SavedStyle\int\,$}\kern-\wd0}\int}

\newcommand{\ip}[2]{\left\langle #1,#2 \right\rangle} 

\newcommand{\tb}{\dashint}

\newcommand{\tp}{\displaystyle\int}

\newcommand{\su}{\mathop{\sup} \limits}

\newcommand{\osc}[2]{\mathop {\operatorname{osc}}\limits_{#1} #2}
\newcommand{\A}{\mathbf A}

\newcommand{\N}{\mathbb N}

\newcommand{\R}{\mathbb R}

\newcommand{\va}{\varphi}

\newcommand{\Om}{\Omega}

\newcommand{\di}{\operatorname{div}}

\newcommand{\ep}{\epsilon}

\newcommand{\f}{\frac}

\begin{document}

\title[Up-to-boundary pointwise gradient estimates]{Up-to-boundary pointwise gradient estimates for very singular quasilinear elliptic equations with mixed data}

\author[T. D. Do]{Tan Duc Do}
\address{Tan Duc Do \\
	University of Economics Ho Chi Minh City}
\email{tanducdo.math@gmail.com}

\author[L. X. Truong]{Le Xuan Truong}
\address{Le Xuan Truong \\
	University of Economics Ho Chi Minh City}
\email{lxuantruong@ueh.edu.vn}

\author[N. N. Trong]{Nguyen Ngoc Trong$^*$}\thanks{$^*$Corresponding author}
\address{Nguyen Ngoc Trong \\
	University of Economics Ho Chi Minh City}
\email{trongnn37@gmail.com}


\arraycolsep=1pt

\begin{abstract}
This paper establishes pointwise estimates up to boundary for the gradient of weak solutions to
a class of very singular quasilinear elliptic equations with
mixed data
\begin{equation*}
\begin{cases}
-\operatorname{div}\left(\A(x,D u)\right)=g-\di f \quad & \mathrm{in} \quad \Om \\ 
u= 0 \quad & \text{on} \ \partial \Omega, 
\end{cases}
\end{equation*}
where $\Omega \subset \mathbb{R}^n$ is sufficiently flat in the sense of Reifenberg.
\end{abstract}

\maketitle

\tableofcontents

\section{Introduction}
Regularity for solutions to $p$-Laplacian equations is classical, yet in itself is an interesting topic. 
Generally speaking, the theory is up to date well-known with various results fully established.
This naturally sparks extensions to more general settings of $p$-Laplacian type equations which have now become an active area of research. 
For a partial overview on its mainstream development, one may take into account pioneering works such as \cite{B83}, \cite{Evan82}, \cite{Iwa83} , \cite{LU68}, \cite{Lewis83}, \cite{Lieberman84}, \cite{Lieberman86}, \cite{Lieberman88}, \cite{To84}, \cite{U77}, \cite{U68}, \cite{N1}, \cite{NP1} and \cite{NP5} (see also their references therein). 
Specifically, these papers discussed in details the interior $C^{1,\alpha}$ regularity for weak solutions to $p$-Laplacian equations with different types of data and to homogeneous quasilinear elliptic equations of the form $-\di(\A(x,u,Du)) = 0$.

In this paper, we pay particular attention to quasilinear elliptic equations formulated in a general form
\[
\di(\A(x,Du))=G,
\]
where $G$ may take on either divergence form, non-divergence form or even a measure value.
Some regularity results for solutions to such equations are readily available.
Namely, interior and global regularities are known for $\di(\A(x,Du))=\di(|F|^{p-2}F)$ with zero
Dirichlet boundary data.
These are due to the seminal works of Byun et al. in \cite{Byun13}, \cite{Byun04}, \cite{Byun07} as well as those of Mengesha et al and Nguyen in \cite{MP11}, \cite{MP12},\cite{N2} in the setting of (weighted or unweighted) Lebesgue and Sobolev spaces. 
When the divergence form $\di(\A(x,Du))=\di(F)$ is considered, interior $W^{1,q}$ estimates were obtained by Nguyen et al. in \cite{NT16}.
Note that these works land themselves in the expected setting of either the ellipticity condition of $\A$ or a requirement of small $BMO$ oscillation in $x$.
Besides, the underlying geometric framework involves, in its most general sense, Reifenberg flat domains $\Omega$.
Analogous results still remain when those settings are weaken or vary in certain manners.
For example, one can take into account \cite{Phuc14} for a weaker hypothesis on $\Omega$ in which a $p$-capacity uniformly thickness condition is used and for solutions belonging to Lorentz spaces. 
Recently \cite{FN} and \cite{Nguyen} studied the global gradient estimates in weighted Morrey spaces, which generalize \cite{NT16}. 
Other results under different assumptions on $\Omega$, the nonlinearity $\A$ and the boundary data can be found in \cite{AP15}, \cite{BCDK17}, \cite{BPS18}, \cite{CM14}, \cite{GW13}, \cite{KZ99},\cite{NP5}.
When a measure-valued $G$ is considered, interior pointwise gradient estimates via Riesz and Wolff potentials were obtained in \cite{Duzamin2} and \cite{Duzamin3} for growth exponents in the ranges $(2-\frac{1}{n}, \infty)$ and $(2,\infty)$ respectively.
Beside quasilinear equations, there is also a growing interest in system settings which are much more involved.
We refer the readers to \cite{DM}, \cite{DHM}, \cite{DHM2}, \cite{FN}, \cite{KM}, \cite{DTT}, \cite{NP1}, \cite{NP3}, \cite{NP4}, \cite{NP5}, \cite{N1} and \cite{N2}  and their references therein for further exploration in this direction.

The aforegoing discussion leads us to consider a gradient estimate up to boundary for the solutions to a very singular quasilinear elliptic equation with mixed data.
More precisely, let $n \in \{2, 3, 4, \ldots\}$ and $\Omega$ be a bounded open subset of $\mathbb{R}^n$.
Let $p \in (1, n)$.
Consider 
\begin{equation}\label{main}
\begin{cases}
-\operatorname{div}(\A(x,D u))=g-\di f \quad & \mathrm{in} \quad \Om, \\ 
u= 0 \quad & \text{on} \ \partial \Omega, 
\end{cases}
\end{equation}
in which 
$$
g\in L^{\frac{p}{p-1}}(\Omega,\R) 
\quad \mbox{and} \quad  
f\in L^{\frac{p}{p-1}}(\Omega,\R^n).
$$ 
The nonlinearity  
$$
\A:\mathbb{R}^n\times \mathbb{R}^n\to \mathbb{R}^n
$$ 
is a Caratheodory function in the sense that $\A(x, z)$ is measurable in $x$ for every $z$ and is continuous in $z$ for a.e.\ $x$. 
Moreover, $\A(x, z)$ is differentiable in $z$ away from the origin for a.e.\ $x$.   
The following {\it growth, ellipticity} and  {\it continuity assumptions} are also imposed on $\A$.
Suppose there exists a $\Lambda \geq 1$ such that
\begin{eqnarray}
&& | \A(x,z)|\le \Lambda \, |z|^{p-1} 
\quad \mbox{and} \quad | D_2 \A(x,z)|\le \Lambda \, |z|^{p-2},
\label{condi1}
\\
&& \langle D_2 \A(x,z) \, \eta,\eta\rangle\geq \Lambda^{-1}  \, |z|^{p-2} \, |\eta|^2,
\label{condi2}
\\
&& |D_2 \A(x,z)-D_2 \A(x,\eta)|\leq  
\left\{
\begin{array}{ll}
\Lambda \, |z|^{p-2} \,  |\eta|^{p-2} \, (|z|^2+|\eta|^2)^{(2-p-\alpha)/2} \, |z-\eta|^\alpha & \mbox{  if } p < 2,
\\
\Lambda \, (|z|^2+|\eta|^2)^{(p-2-\alpha)/2} \, |z-\eta|^\alpha & \mbox{  if } p \ge 2
\end{array}
\right.
\label{condi3}
\\
&& 
\mbox{for some } \alpha\in (0, |2-p|) \mbox{ if } p \ne 2 \mbox{ and } \alpha \in (0,1] \mbox{ if } p=2,
\nonumber
\\
&& |\A(x,z)-\A(x_0,z)|\leq \Lambda \, \omega(|x-x_0|) \, |z|^{p-1}
\label{Dini}
\end{eqnarray}
for all $x, x_0 \in \mathbb{R}^n$ and for all $(z,\eta)\in \mathbb{R}^n\times \mathbb{R}^n\setminus\{(0,0)\}$, where $D_2\A(x,z)$ denotes the Jacobian matrix of $\A$ with respect to the second variable $z\in \R^n\setminus \{0\}.$

As usual the function $\omega: [0,\infty) \longrightarrow [0,1]$ in \eqref{Dini} is required to be non-decreasing and satisfies 
$$
\lim_{r\downarrow 0} \omega(r)=\omega(0)=0
$$ 
as well as the Dini's condition
\begin{equation} \label{Dini int}
\int_{0}^{1}\omega(r)^{\tau_0} \frac{dr}{r}:=Z < \infty,
\end{equation}
where $\tau_0 := \frac{2}{p} \wedge 1$.

\emph{The above conditions on $n$, $p$, $\Omega$, $\A$, $f$ and $g$ (as well as the parameters therein) serve as the background of our paper and will be imposed without being mentioned in all statements throughout.
Further assumptions if required will be explicitly articulated in the corresponding statements.}

Certain nice consequences can be drawn from the above conditions.
The first inequality in \eqref{condi1} and the Caratheodory property together yield 
$$
\A(x,0) = 0 \quad \mbox{for a.e. } x \in \R^n.
$$ 
Meanwhile \eqref{condi2} implies the strict monotonicity condition
\begin{equation} \label{mono}
\ip{\A(x,z)-\A(x,\eta)}{z-\eta}\geq C(n,p,\Lambda) \, \Phi(z,\eta)
\end{equation}
for all $(z,\eta) \in \R^n \times \R^n \setminus \{(0,0)\}$ and for a.e.\ $x \in \R^n$, where $\Phi: \mathbb{R}^n\times \mathbb{R}^n\to \mathbb{R}$ is defined by
\[
\Phi(z,\eta):=
\left\{
\begin{array}{ll}
\left(|z|^2+|\eta|^2\right)^{\frac{p-2}{2}}|z-\eta|^2 & \mbox{if } p \le 2,
\\
|z-\eta|^p & \mbox{if } p > 2
\end{array}
\right.
\]
due to \cite[Lemma 1]{To84}.

A typical model for \eqref{main} is obviously given by the $p$-Laplace equation with mixed data
\[
-\Delta_p \, u:= -{\rm div}(|D u|^{p-2} Du)=g-\di f \quad \text{in~} \Omega.
\]
Note that this special case of ours is very much in line with \cite{NP5} in which the $p$-Laplace equation with mixed data of the form $\di(|F|^{p-2}F)+f$ and homogeneous Dirichlet boundary condition is investigated. Moreover, the parabolic version of \eqref{main} is also considered in \cite{N1} and \cite{N2}.

Our aim here is to derive pointwise estimates for the gradient of solutions to \eqref{main} up to the boundary of $\Om$. 
The main contribution involves the reconstructions of a Holder continuity estimate for the solutions to \eqref{main}, a comparison estimate, an iteration lemma and a Gehring's lemma peculiar to the setting in this paper.

To state our main result, we first need some definitions.

\begin{defn}
	A function $u\in W^{1,p}_0(\Om)$ is a solution to \eqref{main} if
	\[
	\int_{\Omega} \A(x,Du) \cdot D\va \, dx=\int_{\Omega} f \cdot D \va \, dx +\int_{\Omega} g \, \va \, dx
	\]
	for all $ \va \in W^{1,p}_0(\Om).$
\end{defn}

In what follows, for each measurable function $h: \Omega \longrightarrow \R$ denote
$$
(h)_B = \tb_{B} h(x) \, dx = \frac{1}{|B|} \int_{B} h(x) \, dx.
$$
The oscillation of $h$ on a set $A \subset \Omega$ is defined by
\[
\osc{A}{h}
:= \su_{x,y\in A} \big( h(x)-h(y) \big)
= \sup_{x \in A} h(x) - \inf_{x\in A} h(x).
\]
Also set
\[
{\bf F}^R_q(f,g)(x)
:= \int_{0}^{R}\left[\left(\tb_{B_{\rho}(x_0)}|f-(f)_{B_{\rho}(x_0)}|^{q'} \, dx\right)^{\frac{1}{q}} + \rho^{\frac{1}{q-1}} \, \left(\tb_{B_{\rho}(x_0)}|g|^{\frac{nq}{nq-n+q}} \, dx\right)^{\frac{nq-n+q}{(nq-n)q}}   \right] \, \frac{d\rho}{\rho}
\]
for each $q \in (1,n)$ and $R > 0$, where $q'$ is the conjugate index of $q$.

Our first main result is as follows.

\begin{thm}\label{grad1}
	Suppose that $u\in C^1(\Omega)$ solves \eqref{main}. 
	Then 
	\[
	|D u(x)| \leq C \left[ {\bf F}^R_p(f,g)(x) + \left(\tb_{B_{R}(x)}|D u(y)|^{p} \, dy\right)^{\frac{1}{p}}\right]
	\]
	for all ball $B_R(x)\subset\Omega$, where $C = C(n, p,\Lambda,Z)$.
\end{thm}

It is worth mentioning that, in comparison to \cite{Duzamin2} and \cite{Duzamin3}, Theorem 1.2 provides an interior gradient estimate for all $p \in (1,n)$.
The proof of Theorem \ref{grad1} makes use of the following $C^{1,\kappa}$-regularity estimate for the associated homogeneous equation.
This result is also of independent interest.

\begin{thm}\label{theo2}
	Set $\A_0=\A(0,\cdot)$.
	Let $v\in W^{1,p}_{\rm loc}(\Omega)$ be a weak solution of 
	\begin{equation} \label{p-harmonic}
	{\rm div}\, \A_0(D v)=0 \quad \mbox{in } \Omega.
	\end{equation}
	Then there exist constants $C\geq 1$, $\kappa \in (0,1)$ and $\sigma_0 \in (0,1/2]$, all of which depend on $n$ and $p$ only such that                                       
	\begin{equation}\label{holder}
	\osc{B_{\delta r}}{Dv} \leq C \, \delta^\kappa \,  \tb_{B_r}|Dv-(Dv)_{B_r}| \, dx 
	\end{equation}
	for all $\delta \in (0,\sigma_0]$ and for all ball $B_r \subset \Omega$.
\end{thm}

Next we state the pointwise estimate up to the boundary of $\Omega$ for a solution $u$ of \eqref{main}.
This is possible provided that $\partial\Omega$ is sufficiently flat in the sense of Reifenberg.

\begin{defn}
Let $\delta\in (0,1)$ and $R_0>0$.
Then $\Omega$ is called a $(\delta,R_0)$-Reifenberg flat domain if for all $x\in\partial \Omega$ and $r\in(0,R_0]$ there exists a system of coordinates $\{z_1,z_2,...,z_n\}$, depending on $r$ and $x$, in which one has $x=0$ and 
	$$
	B_r(0)\cap \{z_n>\delta r\}\subset B_r(0)\cap \Omega\subset B_r(0)\cap\{z_n>-\delta r\}.
	$$
\end{defn}

This class of domains first appeared in the work of Reifenberg \cite{55Re} in the context of Plateau problems.
The class is broad enough to include $C^1$ domains, Lipschitz domains with sufficiently small Lipschitz constants and certain fractal domains. 
Its significance has been solidified in the theory of minimal surfaces and free boundary problems. 
For many interesting properties of Reifenberg flat domains, we refer the readers to \cite{55KeTo1} and \cite{55KeTo2}.

Our pointwise estimate up to the boundary of $\Omega$ is as follows.

\begin{thm} \label{boundary}  
	Suppose that $u$ is a solution of \eqref{main}. 
	Then there exists a $C = C(n,p,\Lambda) > 0$ such that 
	\begin{align}\label{ine3}
	|D u(x)|\leq C d(x)^{-\zeta/p} \, {\bf F}^{2\operatorname{diam}(\Omega)}_p(f,g)(x)
	\end{align}
	for a.e.\ $x\in \Omega$ and for all $\zeta \in (0,n)$, where $d(x) := \mathrm{dist}(x,\partial\Omega)$.
\end{thm}    
We remark that it is in general not possible to take $\zeta=0$ in \eqref{ine3} due to a potential irregularity of $\Omega$. Beside, \eqref{ine3} is proved in \cite{NP3} when $f\equiv 0$ and $g$ is a finite signed measure
in $\Omega$.

The outline of the paper is as follows. 
Proposition \ref{theo2} is proved in Section \ref{holder est}.
Section \ref{interior} discusses a comparison estimate and an iteration lemma which are then used in the proof of the interior pointwise gradient estimate - Theorem \ref{grad1}. 
The global pointwise gradient estimate - Theorem \ref{boundary} is obtained in Section \ref{global}.

{\bf Notation.} \quad
Throughout the paper the following set of notation is used without mentioning.
Set $\N = \{0, 1, 2, 3, \ldots\}$ and $\N^* = \{1, 2, 3, \ldots\}$.
For all $a, b \in \R$, $a \wedge b = \min\{a,b\}$ and $a \vee b = \max\{a,b\}$.
For all ball $B \subset \R^d$ we write $w(B) := \int_B w$.
The constants $C$ and $c$ are always assumed to be positive and independent of the main parameters whose values change from line to line.
Given a ball $B = B_r(x)$, we let $t B = B_{tr}(x)$ for all $t > 0$.
The conjugate index of $p \in [1,\infty)$ is denoted by $p'$. 

\section{A $C^{1,\kappa}$-regularity estimate}
\label{holder est}

We start by proving Theorem \ref{theo2}.
This paves way for the proof of Theorem \ref{grad1} later, but is also of independent interest.

\noindent 
\begin{pot}{theo2}
We follow the arguments used in \cite[Proof of Theorem 3.2]{KM} closely.

By \cite[Proposition 3.3]{B83}, for all $\delta \in (0,1)$ there exists a constant $C = C(n,p,\delta) > 0$ such that 
\[
\sup_{B_{\delta r}} |Dv| \le C \, \left( \tb_{B_r} |Dv|^p \, dx \right)^{1/p},
\]
which in turn yields
\begin{equation} \label{sup ineq}
\sup_{B_{\delta r}} |Dv| \le C \, \tb_{B_r} |Dv| \, dx
\end{equation}
due to \cite[Lemmas 2.1 and 2.4]{Phuc14c}.

With \eqref{sup ineq} in mind, one can then apply standard arguments to obtain a $C^{1,\kappa}$ a priori local estimate (cf.\ \cite[Theorem 2, p.15]{Man1986}).
Specifically, there exist constants $C_0 = C_0(n,p) \ge 1$ and $\beta_0 = \beta_0(n,p) \in (0,1)$ such that  
\begin{equation} \label{osc 0}
\osc{B_{\delta r}}{Dv} \le C_0 \, \delta^{\beta_0} \, \tb_{B_r} |Dv| \, dx
\end{equation}
for all $\delta \in (0,2/3)$.

Now we consider two cases.

\noindent
{\bf Case 1}: Suppose
\[
\tb_{B_r} |Dv - (Dv)_{B_r}| \, dx \ge \theta_0 \, |(Dv)_{B_r}|,
\]
where 
\begin{equation} \label{thesig0}
\theta_0 := \frac{\sigma_0^n}{40} \in (0,1)
\quad \mbox{and} \quad
\sigma_0 := \left( \frac{1}{160 C_0} \right)^{1/\beta_0} \in \left(0,\frac{1}{160}\right).
\end{equation}
Then \eqref{osc 0} gives
\begin{eqnarray*}
\osc{B_{\delta r}}{Dv}
&\le& C_0 \, \delta^{\beta_0} \, \tb_{B_r} |Dv - (Dv)_{B_r}| \, dx
+ C_0 \, \delta^{\beta_0} \, |(Dv)_{B_r}| 
\\
&\le& C_0 \, \left(1 + \frac{1}{\theta_0} \right) \, \delta^{\beta_0} \, \tb_{B_r} |Dv - (Dv)_{B_r}| \, dx
\end{eqnarray*}
for all $\delta \in (0,2/3)$.
Therefore \eqref{holder} holds with $\sigma_0 = 1/2$.

\noindent
{\bf Case 2}: Suppose
\[
\tb_{B_r} |Dv - (Dv)_{B_r}| \, dx < \theta_0 \, |(Dv)_{B_r}|,
\]
where $\theta_0$ is as in \eqref{thesig0}.
To avoid triviality, assume further that $|(Dv)_{B_r}| > 0$.
Then 
\begin{equation} \label{c2pre}
\tb_{B_r} |Dv| \, dx \le 2 \, |(Dv)_{B_r}|,
\end{equation}
by the triangle inequality.
Combining this estimate with \eqref{osc 0} and the definition of $\sigma_0$ leads to
\begin{equation} \label{c2}
\osc{B_{2\sigma_0 r}}{Dv} 
\le 4C_0\sigma_0^{\beta_0} \, |(Dv)_{B_r}|
\le |(Dv)_{B_r}|/20.
\end{equation}

At the same time, one also has
\[
\tb_{B_{2\sigma_0 r}} |Dv - (Dv)_{B_r}| \, dx
\le (2\sigma_0)^{-n} \, \tb_{B_r} |Dv - (Dv)_{B_r}| \, dx
\le (2\sigma_0)^{-n} \, \theta_0 \, |(Dv)_{B_r}|.
\]
This ensures the existence of an $\tilde{x} \in B_{2\sigma_0 r}$ such that 
\begin{equation} \label{x tilde}
|Dv(\tilde{x}) - (Dv)_{B_r}| \le |(Dv)_{B_r}|/20.
\end{equation}
Now \eqref{c2} and \eqref{x tilde} together imply
\begin{equation} \label{c2a}
|Dv(x) - (Dv)_{B_r}| \le |(Dv)_{B_r}|/10
\end{equation}
for all $x \in B_{2\sigma_0 r}$.

On the other hand, we infer from \eqref{sup ineq} and \eqref{c2pre} that there exists a $c = c(n,p) > 0$ such that
\begin{equation} \label{c2b}
|Dv(x)| \le c \, \tb_{B_r} |Dv| \le 2c \, |(Dv)_{B_r}|
\end{equation}
for all $x \in B_{2r/3}$.

Combining \eqref{c2a} and \eqref{c2b} yields
\begin{equation} \label{two-way ineq}
|(Dv)_{B_r}|/ c \le |Dv(x)| \le c \, |(Dv)_{B_r}|
\end{equation}
for some $c = c(n,p)$ and for all $x \in B_{2\sigma_0 r}$.

Next we take \eqref{osc 0} and \eqref{c2pre} into account to derive
\begin{equation} \label{c2c}
|Dv(x) - Dv(y)| \le c(n,p) \, |(Dv)_{B_r}| \, |x-y|^{\beta_0}
\end{equation}
for all $x, y \in B_{r/2}$.
By partially differentiating \eqref{p-harmonic} and performing a scaling on the coefficients, the partial derivatives $w_i := \partial_{x_i} v$ satisfy
\[
-{\rm div} \left( \frac{\A_0'(Dv)}{|(Dv)_{B_r}|^{p-2}} \, Dw_i \right) = 0,
\]
which is a uniformly elliptic equation due to \eqref{condi1}, \eqref{condi2} and \eqref{two-way ineq}. 
Furthermore we deduce from \eqref{condi3} and \eqref{c2c} that there exists a $c = c(n,p) > 0$ satisfying
\[
\left|\frac{\A_0'(Dv)(x)}{|(Dv)_{B_r}|^{p-2}} - \frac{\A_0'(Dv)(y)}{|(Dv)_{B_r}|^{p-2}} \right| \le c \, |x-y|^{\alpha \, \beta_0}
\]
for all $x, y \in B_{2\sigma_0 r}$, where $\alpha$ is given by \eqref{condi3}. 

Consequently a standard result provides that $w_i$ is locally $\gamma$-Holder continuous in $B_{\sigma_0 r}$ for all exponent $\gamma < 1$.
Particularly, there exists a $c = c(n,p) > 0$ such that
\[
\osc{B_{\delta r}}{w_i}  
\le c \, \delta^{\beta_0} \, \tb_{B_{2\sigma_0 r}} |w_i - (w_i)_{B_{\sigma_0 r}}| \, dx 
\le 2 c \sigma_0^{-n} \, \delta^{\beta_0} \, \tb_{B_r} |w_i - (Dv)_{B_r}| \, dx
\] 
for all $\delta \in (0,\sigma_0)$, where the last step is justified by \cite[(2.1)]{KM}.
Whence the claim follows for all $\delta \in (0,\sigma_0]$.

This completes our proof.
\end{pot}

As a consequence of Theorem \ref{theo2} , we obtain the following lemma.
\begin{lem}\label{hol-Du}
	Set $\A_0=\A(0,\cdot)$.
	Let $v\in W^{1,p}_{\rm loc}(\Omega)$ be a weak solution of 
	\[
	{\rm div}\, \A_0(D v)=0 \quad \mbox{in } \Omega.
	\]
	Then there exist constants $C \ge 1$, $\kappa \in (0,1)$ and $\sigma_0 \in(0,1/2]$, all of which depend on $n$ and $p$ only, such that
	\[
	\left(\tb_{B_{\rho}(x_0)} |D v - (D v)_{B_{\rho}(x_0)}|^q \, dx\right)^{1/q} \leq C \left(\frac{\rho}{R}\right)^{\kappa} \left(\tb_{B_{R}(x_0)} |D v - (D v)_{B_{R}(x_0)}|^q dx\right)^{1/q}
	\]
	for all $q \in [1,\infty)$, $B_R(x_0)\subset\Omega$ and $0<\rho<\sigma_0 R$.
\end{lem}

\section{Interior pointwise gradient estimates} \label{interior}

The main goal of this section is to prove Theorem \ref{grad1}. 
To this end, we first derive several preliminary results.

Let $u\in W_{0}^{1,p}(\Omega)$ be a solution of \eqref{main}.
Suppose that $B_{2r}=B_{2r}(x_0)\Subset\Omega$.  

In this whole section, let $w\in  u+ W_{0}^{1,p}(B_{2r})$ be the unique solution to the  problem 
\begin{equation}
\label{thuannhat1}\left\{ \begin{array}{rcl}
- \operatorname{div}\left( {\A(x,D w)} \right) &=& 0 \quad {\rm in} \quad B_{2r}, \\ 
w &=& u \quad {\rm on} \quad \partial B_{2r}
\end{array} \right.
\end{equation}
and let $v\in W_0^{1,p}(B_r(x_0))$ be the  unique solution of 
\begin{equation*}
\left\{ \begin{array}{rcl}
- \operatorname{div}\left( {\A(x_0,D v)} \right) &=& 0 \quad {\rm in} \quad B_{r}, \\ 
v &=& w \quad {\rm on} \quad \partial B_{r}.  
\end{array} \right.
\end{equation*}

Several convenient properties of $w$ and $v$ can immediately listed here.

As a result of \cite[Lemmas 2.1 and 2.4]{Phuc14c} (also cf.\ \cite[(3.3)]{NP3}), there exists a $C = C(n,p,\Lambda) > 0$ such that 
\begin{align}\label{es29}
\|D v\|_{L^\infty(B_{r/2})}\leq C \left(\tb_{B_r}|D v|^{t}dx\right)^{1/t}
\end{align}
for all $t>0$.

We also have an estimate for the difference  $D v-D w$ as follows.
There exists a $C = C(n,p,\Lambda) > 0$ such that 
\begin{equation}\label{es23}
\tb_{B_{r}}|D v-D w|^p \, dx 
\leq C \omega(r)^{\tau_0 p} \tb_{B_{r}}|D w|^p \, dx,
\end{equation}
whose proof can be found in \cite[(4.35)]{Duzamin2} and \cite[Lemma 3.4]{Duzamin3} for $p \in (1,2]$ and $p \in (2,n)$ respectively.

%
%

\subsection{New comparison estimate}

Now we construct a comparison estimate that serves to prove Theorem \ref{grad1}.
It is new as it adapts our general setting proposed in this paper.

\begin{lem}\label{111120149}
	
	Let $u \in W_0^{1,p}(\Omega)$ be a solution to \eqref{main}. Then there exists a constant $C = C(n,p,\Lambda) > 0$ such that
	\begin{align*}
	& \tb_{B_{2r}}|Du-Dw|^pdx\\
	& \leq C \left[\inf_{m \in \R^n} \tb_{B_{2r}}|f-m|^{p'}dx + r^{p'}\left(\tb_{B_{2r}}|g|^{\frac{np}{np-n+p}}dx\right)^{\frac{np-n+p}{np-n}} \right]\\
	& \quad  +  C \left[\inf_{m \in \R^n} \left(\tb_{B_{2r}}|f-m|^{p'}dx\right)^{p-1} + r^p\left(\tb_{B_{2r}}|g|^{\frac{np}{np-n+p}}dx\right)^{\frac{np-n+p}{n}} \right] \left(\tb_{B_{2r}}|D u|^{p}dx\right)^{2-p}
	\end{align*}
	for $p \in (1,2)$ and
	\begin{align*}
	\tb_{B_{2r}}|D u-D w|^pdx
	\leq C \left[\inf_{m \in \R^n} \tb_{B_{2r}}|f-m|^{p'}dx + r^{p'}\left(\tb_{B_{2r}}|g|^{\frac{np}{np-n+p}}dx\right)^{\frac{np-n+p}{np-n}} \right]
	\end{align*}
	for $p\in [2,n)$.
\end{lem}
\begin{proof}
	Let $m \in \R^n$.
	Then
	\begin{equation}\label{t1}
	\begin{aligned}
	\tp_{B_{2r}} \big[ A(x,Du)-A(x,Dw) \big] \cdot D\va \, dx=\int_{B_{2r}} (f-m) \cdot D\va \, dx+\int_{\R^n} g \, \va \, dx
	\end{aligned}
	\end{equation}
	for all $\va \in W^{1,p}_0(B_{2r})$.
Choosing $\va =u-w$ as a test function in \eqref{t1} gives
	\begin{align*}
	\tp_{B_{2r}} \big[ A(x,Du)-A(x,Dw) \big] \cdot D(u-w) dx
	=\int_{B_{2r}} (f-m) \cdot D (u-w) dx+\int_{\R^n} g \, (u-w) \, dx.
	\end{align*}
	Using \eqref{mono} and Holder's inequality this leads to
	\begin{align*}
	\tb_{B_{2r}} \Phi(Dw, Du) \, dx
	&\le C(n,p,\Lambda) \left(\tb_{B_{2r}}|f-m|^{p'}dx\right)^{1/p'} \left(	\tb_{B_{2r}}|D u-D w|^{p}dx\right)^{\frac{1}{p}}
	\\
	& \quad {} + C(n,p,\Lambda)\left(\tb_{B_{2r}}|g|^{\frac{np}{np-n+p}}dx\right)^{\frac{np-n+p}{np}} \left(	\tb_{B_{2r}}| u- w|^{\frac{np}{n-p}}dx\right)^{\frac{n-p}{np}}.
	\end{align*}	
	Next we apply Sobolev's inequality to obtain
	\begin{equation*}
	\begin{aligned}
	\left(	\tb_{B_{2r}}| u- w|^{\frac{np}{n-p}}dx\right)^{\frac{n-p}{np}}
	& \le C(n,p) \, r\left(	\tb_{B_{2r}}|D u- D w|^{p}dx\right)^{1/p}.
	\end{aligned}
	\end{equation*}	
	Consequently,	
	\begin{equation}\label{t9}
	\begin{aligned}
	\tb_{B_{2r}} \Phi(Dw, Du) \, dx 
	& \le C(n,p,\Lambda) \left[\left(\tb_{B_{2r}}|f-m|^{p'}dx\right)^{1/p'} + r\left(\tb_{B_{2r}}|g|^{\frac{np}{np-n+p}}dx\right)^{\frac{np-n+p}{np}} \right]\\
	& \quad \times \left(	\tb_{B_{2r}}|D u-D w|^{p}dx\right)^{\frac{1}{p}}.
	\end{aligned}
	\end{equation}

Now we consider two cases.

{\bf Case 1}: Suppose $p \in (1,2)$.
Then rewrite
\begin{equation} \label{decom}
	\begin{aligned}
	|D u -D w|^p
	& =\left(|D u|+|D w|\right)^{\frac{p(p-2)}{2}}|D u -D w|^p\left(|D u|+|D w|\right)^{\frac{p(2-p)}{2}}\\
	& \leq \left(|D u|+|D w|\right)^{\frac{p(p-2)}{2}}|D u -D w|^p\left(|D u|+|D u-D w|\right)^{\frac{p(2-p)}{2}}\\
	& \leq \left(|D u|+|D w|\right)^{\frac{p(p-2)}{2}}|D u -D w|^p|D u|^{\frac{p(2-p)}{2}}\\
	& \quad +\left(|D u|+|D w|\right)^{\frac{p(p-2)}{2}}|D u -D w|^p\left(|D u-D w|\right)^{\frac{p(2-p)}{2}}.
	\end{aligned}
\end{equation}
Using Young’s inequality in the form
	\[
	ab^{\frac{2-p}{2}}\leq \frac{p \ep^{\frac{p-2}{p}}a^{\frac{2}{p}}}{2}+\frac{(2-p)\ep b}{2}
	\]
with an appropriate $\epsilon > 0$, we arrive at
	\begin{align*}
	|D u -D w|^p
	& \leq \left(|D u|+|D w|\right)^{\frac{p(p-2)}{2}}|D u -D w|^p|D u|^{\frac{p(2-p)}{2}} 
	\\
	& \quad + C(p) \, (|D w|+|D u|)^{p-2}|D (u-w)|^2+\frac{1}{2}|D u -D w|^p,
	\end{align*}
	whence it follows from \eqref{decom} that
	\begin{align*}
	|D u -D w|^p
	& \leq C(p) \, (|D w|+|D u|)^{p-2}|D (u-w)|^2\\
	& \quad +C(p) \, \left(|D u|+|D w|\right)^{\frac{p(p-2)}{2}}|D u -D w|^p|D u|^{\frac{p(2-p)}{2}}.
	\end{align*}
	By integrating both sides of the above estimate on $B_{2r}$ and the apply Holder’s inequality with exponents $\frac{2}{p}$ and $\frac{2}{2-p}>1$, we obtain
	\begin{equation}\label{t5}
	\begin{aligned}
	\tb_{B_{2r}}|D (u-w)|^pdx
	& \leq C(p) \tb_{B_{2r}}(|D w|+|D u|)^{p-2}|D (u-w)|^2dx\\
	& \quad+  C(p) \left(\tb_{B_{2r}}(|D w|+|D u|)^{p-2}|D (u-w)|^2dx\right)^{\frac{p}{2}}\\
	& \hspace{2cm} \times \left(\tb_{B_{2r}}|D u|^{p}dx\right)^{\frac{2-p}{2}}.
	\end{aligned}
	\end{equation}
	
	Combining \eqref{t9} with \eqref{t5} yields 
	\begin{align*}
	\tb_{B_{2r}}|D (u-w)|^pdx
	& \leq C(n,p,\Lambda) \left[\left(\tb_{B_{2r}}|f-m|^{p'}dx\right)^{1/p'} + r\left(\tb_{B_{2r}}|g|^{\frac{np}{np-n+p}}dx\right)^{\frac{np-n+p}{np}} \right]\\
	& \hspace{2.2cm} \times \left(	\tb_{B_{2r}}|D u-D w|^{p}dx\right)^{\frac{1}{p}}\\
	& \quad+  C(n,p,\Lambda) \left[\left(\tb_{B_{2r}}|f-m|^{p'}dx\right)^{1/p'} + r\left(\tb_{B_{2r}}|g|^{\frac{np}{np-n+p}}dx\right)^{\frac{np-n+p}{np}} \right]^{p/2}\\
	& \hspace{2.5cm} \times \left(	\tb_{B_{2r}}|D u-D w|^{p}dx\right)^{\frac{1}{2}} \left(\tb_{B_{2r}}|D u|^{p}dx\right)^{\frac{2-p}{2}}.
	\end{align*}

	Another application of Young's inequality yields
	\begin{align*}
	& \tb_{B_{2r}}|D (u-w)|^pdx\\
	& \leq C(n,p,\Lambda) \left[\tb_{B_{2r}}|f-m|^{p'}dx + r^{p'}\left(\tb_{B_{2r}}|g|^{\frac{np}{np-n+p}}dx\right)^{\frac{np-n+p}{np-n}} \right]\\
	& \quad  +  C(n,p,\Lambda) \left[\left(\tb_{B_{2r}}|f-m|^{p'}dx\right)^{p-1} + r^p\left(\tb_{B_{2r}}|g|^{\frac{np}{np-n+p}}dx\right)^{\frac{np-n+p}{n}} \right] \left(\tb_{B_{2r}}|D u|^{p}dx\right)^{2-p}
	\end{align*}
	as required.

{\bf Case 2}: Suppose $p \in [2,n)$.
Starting from \eqref{t9}, we have
\begin{align*}
\tb_{B_{2r}} |D (u-w)|^pdx 
& \le C(n,p,\Lambda) \left[\left(\tb_{B_{2r}}|f-m|^{p'}dx\right)^{1/p'} + r\left(\tb_{B_{2r}}|g|^{\frac{np}{np-n+p}}dx\right)^{\frac{np-n+p}{np}} \right]\\
& \quad \times \left(	\tb_{B_{2r}}|D u-D w|^{p}dx\right)^{\frac{1}{p}}.
\\
& \leq \frac{1}{2} \tb_{B_{2r}}|D u-D w|^{p}dx\\
& \quad + C(n,p,\Lambda) \left[\tb_{B_{2r}}|f-m|^{p'}dx + r^{p'}\left(\tb_{B_{2r}}|g|^{\frac{np}{np-n+p}}dx\right)^{\frac{np-n+p}{np-n}} \right]
\end{align*}
as required, where we used Young's inequality in the second step.

This completes the proof.
\end{proof}

\subsection{Iteration lemma}
Another key ingredient in the proof of Theorem \ref{grad1} is the iteration lemma given by Proposition \ref{intA} below.

In what follows, define
$$
\mathbf{I}(\rho)=\mathbf{I}(x_0,\rho):=	\left(\tb_{B_{\rho}}|D u-\left(D u \right)_{B_{\rho}}|^pdx\right)^{1/p}
$$
for a ball $B_\rho=B_\rho(x_0)\subset\Omega$.

\begin{prop} \label{intA}  
Suppose  that $u\in W_{0}^{1,p}(\Omega)$ is a solution of \eqref{main}. 
Then there exist constants $\kappa \in (0,1)$ and $\sigma_0\in (0,1/2]$ such that  
	\begin{align*}
	& \mathbf{I}(\delta r)\\
	 \leq \,\, & C_0 \, \delta^{\kappa/p} \, \mathbf{I}(r)
	 + C_0 \, \delta^{-n/p} \left[\inf_{m \in \R^n} \left(\tb_{B_{r}}|f-m|^{p'}dx\right)^{1/p} + r^{1/(p-1)}\left(\tb_{B_{r}}|g|^{\frac{np}{np-n+p}}dx\right)^{\frac{np-n+p}{(np-n)p}}  \right]\\
	&  +  C' \, \delta^{-n/p} \left[\inf_{m \in \R^n}\left(\tb_{B_{r}}|f-m|^{p'}dx\right)^{1/p'} + r\left(\tb_{B_{r}}|g|^{\frac{np}{np-n+p}}dx\right)^{\frac{np-n+p}{np}}\right] 
\left(\tb_{B_{r}}|D u|^{p}dx\right)^{(2-p)/p}
\\
& {} + C_0 \, \delta^{-n/p} \, \omega(r)^{\tau_0} \, \left(\tb_{B_{r}}|D u|^p dx\right)^{1/p}
	\end{align*}
	for all $\delta \in (0,\sigma_0]$ and $B_{2r}(x_0)\Subset\Omega$, where $C_0 = C_0(n,p,\Lambda)$ and $C' = C'(n,p,\Lambda)$.
\end{prop}

\begin{rem} \label{rem1}
If $p \in [2,n)$ then we may choose $C' = 0$ in Proposition \ref{intA}.
\end{rem}

\begin{proof}
	By virtue of Lemma \ref{hol-Du}, there exist constants $C \ge 1$, $\kappa \in (0,1)$ and $\sigma_0 \in(0,1/2]$, all of which depend on $n$ and $p$ only, such that
	\begin{align*}
	\tb_{B_{\delta r}}|D u-\left(D u\right)_{B_{\delta r}}|^pdx
	&\leq
	C\tb_{B_{\delta r}}|D w-\left(D w\right)_{B_{\delta r}}|^pdx+C	\tb_{B_{\delta r}} |D u-D w|^{p}dx
	\\
	&\leq
	C	\delta^{\kappa}\tb_{B_{r}}|D w-\left(D w\right)_{B_{r}}|^pdx+C\delta^{-n}	\tb_{B_{r}} |D u-D w|^{p}dx
	\\
	&\leq
	C\delta^{\kappa}\tb_{B_{r}}|D u-\left(D u\right)_{B_{r}}|^pdx+C\delta^{-n}	\tb_{B_{r}} |D u-D w|^{p}dx
	\end{align*}
	for all $\delta \in (0,\sigma_0]$.
	
	Next let $m \in \R^n$.
	We aim to bound the second term on the right-hand side of the above estimate.
	Observe that 
	\begin{align*}
	\tb_{B_{r}} |D u-D v|^{p}dx 
	& \le C \, \tb_{B_{r}} |D u-Dw|^{p} \, dx + \tb_{B_{r}} |Dw-D v|^{p} \, dx
	\\
	& \le C \, \tb_{B_{r}} |D u-Dw|^{p} \, dx +  C \omega(r)^{\tau_0 p} \tb_{B_{r}}|D w|^p \, dx 
	\\
	& \le C \, \tb_{B_{r}} |D u-Dw|^{p} \, dx +  C \omega(r)^{\tau_0 p} \tb_{B_{r}}|D u|^p \, dx,
	\end{align*}
	where we used \eqref{es23} in the second step and the fact that $\omega \le 1$ in the third step.
	Here $C = C(n,p,\Lambda)$.
		
It remains to bound the term 
\[
\tb_{B_{r}} |D u-Dw|^{p} \, dx,
\]
which can be done using Lemma \ref{111120149}.

This completes our proof.
\end{proof}

\subsection{Proof of Theorem \ref{grad1}}

We are now we are ready to prove Theorem \ref{grad1}.

\begin{pot}{grad1}
Let $B_R(x_0) \subset \Omega$.
Set $\epsilon \in (0,\sigma_0/2) \subset (0,1/4)$ be sufficiently small so that $
C_0\epsilon^{\kappa}\leq \frac{1}{4}$, where $C_0$, $\kappa$ and $\sigma_0$ are the constants given by Proposition \ref{intA}.

For all $j \in \N$ set 
\[
R_j=\epsilon^{j} R,
\quad
B_j=B_{2R_j}(x_0),
\quad \mathbf{I}_j=\mathbf{I}(R_j)
\quad \mbox{and} \quad
T_j:=\left(\tb_{B_{j}} |D u|^{p} dx\right)^{1/p}.
\]

It suffices to show that
\begin{equation}\label{maines}
\begin{aligned}
& |D u(x_0)| \\
&\le C \, T_{j_0} + C  \int_{0}^{2R_{1}}\left[\tb_{B_{\rho}(x_0)}|f-\beta|^{p'}dx + \rho^{p'}\left(\tb_{B_{\rho}(x_0)}|g|^{\frac{np}{np-n+p}}dx\right)^{\frac{np-n+p}{np-n}} \right] \, \frac{d\rho}{\rho},
\end{aligned}
\end{equation}
where $C = C(n, p,\Lambda,Z)$.

Let $\beta \in \R^n$.
Then Proposition \ref{intA} gives 
\begin{align*}
\mathbf{I}_{j+1}
& \leq \frac{1}{4}\mathbf{I}_{j} + C_1\left[\left(\tb_{B_{j}}|f-\beta|^{p'}dx\right)^{1/p} + R_j^{1/(p-1)}\left(\tb_{B_{j}}|g|^{\frac{np}{np-n+p}}dx\right)^{\frac{np-n+p}{(np-n)p}} \right]\\
& \quad  +  C_2 \left[\left(\tb_{B_{j}}|f-\beta|^{p'}dx\right)^{1/p'} + R_j\left(\tb_{B_{j}}|g|^{\frac{np}{np-n+p}}dx\right)^{\frac{np-n+p}{np}} \right] \, T_j^{2-p}
\\
& \quad + C_1 \, \omega(R_j)^{\tau_0} \, T_j,
\end{align*}
where $C_j = C_j(n,p,\Lambda,\epsilon)$ for $j \in \{1,2\}$.
Moreover, we may take $C_2 = 0$ in the case $p \in [2,n)$.

Let $j_0, m \in \N$ be such that $j_0 \ge 2$ and $m \ge j_0+1$, whose appropriate values will be chosen later.
Summing the above estimate up over $j\in \{j_0, j_0+1, \ldots,m-1\}$, we obtain 
\begin{equation}\label{z1}
\begin{aligned}
\sum_{j=j_0}^{m}\mathbf{I}_{j}
&\le \frac{4}{3} \, \mathbf{I}_{j_0}+ \frac{4}{3} \, C_1 \sum_{j=j_0}^{m-1}\left[\left(\tb_{B_{j}}|f-\beta|^{p'}dx\right)^{1/p} + R_j^{1/(p-1)}\left(\tb_{B_{j}}|g|^{\frac{np}{np-n+p}}dx\right)^{\frac{np-n+p}{(np-n)p}}  \right]\\
&\quad  + \frac{4}{3} \, C_2 \sum_{j=j_0}^{m-1}\left[\left(\tb_{B_{j}}|f-\beta|^{p'}dx\right)^{1/p'} + R_j\left(\tb_{B_{j}}|g|^{\frac{np}{np-n+p}}dx\right)^{\frac{np-n+p}{np}}\right] \, T_j^{2-p}
\\
& \quad + \frac{4}{3} \, C_1 \, \sum_{j=j_0}^{m-1}\omega(R_j)^{\tau_0} \, T_j.
\end{aligned}
\end{equation}
Observe that 
$$
\sum_{j=j_0}^{m}\mathbf{I}_{j}
\ge \epsilon^n \sum_{j=j_0}^{m}|\left(D u \right)_{B_{j+1}}-\left(D u\right)_{B_{j}}|
\ge \epsilon^n \, |\left(D u \right)_{B_{m+1}}-\left(D u \right)_{B_{j_0}}|.
$$
Therefore \eqref{z1} implies 
\begin{equation}\label{es1}
\begin{aligned}
&|\left(D u\right)_{B_{m+1}}|  \\
&\le \frac{4}{3} \, \epsilon^{-n} \, \mathbf{I}_{j_0}+ |\left(D u\right)_{B_{j_0}}|
\\
& \quad + \frac{4}{3} \, C_1 \, \epsilon^{-n}  \sum_{j=j_0}^{m-1}\left[\left(\tb_{B_{j}}|f-\beta|^{p'}dx\right)^{1/p} + R_j^{1/(p-1)}\left(\tb_{B_{j}}|g|^{\frac{np}{np-n+p}}dx\right)^{\frac{np-n+p}{(np-n)p}}  \right]\\
&\quad  + \frac{4}{3} \, C_2 \, \epsilon^{-n}  \sum_{j=j_0}^{m-1}\left[\left(\tb_{B_{j}}|f-\beta|^{p'}dx\right)^{1/p'} + R_j\left(\tb_{B_{j}}|g|^{\frac{np}{np-n+p}}dx\right)^{\frac{np-n+p}{np}} \right] \, T_j^{2-p}
\\
& \quad + \frac{4}{3} \, C_1 \, \epsilon^{-n}  \, \sum_{j=j_0}^{m-1}\omega(R_j)^{\tau_0} \, T_j.
\end{aligned}
\end{equation}	

Estimating between an integral and its partial sum reveals that
\begin{equation}\label{es0}
\begin{aligned}
& \sum_{j=j_0}^{m-1}\left[\left(\tb_{B_{j}}|f-\beta|^{p'}dx\right)^{1/p} + R_j^{1/(p-1)}\left(\tb_{B_{j}}|g|^{\frac{np}{np-n+p}}dx\right)^{\frac{np-n+p}{(np-n)p}} \right]\\
& \le \varepsilon^{-1}\int_{0}^{2R_{1}}\left[\left(\tb_{B_{\rho}(x_0)}|f-\beta|^{p'}dx\right)^{1/p} + \rho^{1/(p-1)}\left(\tb_{B_{\rho}(x_0)}|g|^{\frac{np}{np-n+p}}dx\right)^{\frac{np-n+p}{(np-n)p}} \right] \, \frac{d\rho}{\rho},
\end{aligned} 
\end{equation}

and 
\begin{equation}\label{es00}
\begin{aligned}
& \sum_{j=j_0}^{m-1}\left[\left(\tb_{B_{j}}|f-\beta|^{p'}dx\right)^{1/p'} + R_j\left(\tb_{B_{j}}|g|^{\frac{np}{np-n+p}}dx\right)^{\frac{np-n+p}{np}}  \right]\\
& \le  \varepsilon^{-1} \int_{0}^{2 R_{1}}\left[\left(\tb_{B_{\rho}(x_0)}|f-\beta|^{p'}dx\right)^{1/p'} + \rho\left(\tb_{B_{\rho}(x_0)}|g|^{\frac{np}{np-n+p}}dx\right)^{\frac{np-n+p}{np}} \right] \, \frac{d\rho}{\rho}.
\end{aligned} 
\end{equation}

In what follows, choose a $j_0 = j_0(\epsilon,C_1,Z,\Omega)$ such that 
\[
\frac{8}{3} \, C_1 \, \epsilon^{-2n} \sum_{j=j_0}^\infty \omega(R_j)^{\tau_0} < \frac{1}{10},
\]
where $C_1$ is given in \eqref{es1}.
Note that this choice is possible due to \eqref{Dini int}.

Now we consider three cases.

{\bf Case 1}: Suppose $|D u(x_0)|\leq  T_{j_0}$.
Then \eqref{maines} trivially follows.

{\bf Case 2}: Suppose there exists a $j_1 \in \N$ such that $j_1 \ge j_0$ and 
\begin{equation}\label{co1}
T_j\leq |D u(x_0)|
\quad \mbox{and} \quad 
|D u(x_0)|<T_{j_1+1}
\end{equation}
for all $j \in \{j_0, j_0+1,\ldots,j_1\}$.

Then
\begin{align*}
|D u(x_0)|& <	\left(\tb_{B_{j_1 +1}} |D u|^{p}dx\right)^{1/p} \leq \mathbf{I}_{j_1+ 1}+ |\left(D u\right)_{B_{j_1 +1}}| \leq  \epsilon^{-n} \, \left( \mathbf{I}_{j_1}+ |\left(D u\right)_{B_{j_1 +1}}| \right).
\end{align*}

Now applying \eqref{z1} and \eqref{es1} with $m=j_1$ and then using  \eqref{es0}, \eqref{es00} and \eqref{co1}
we derive 
\begin{align*}
& |D u(x_0)|\\
&\le \frac{8}{3} \, \epsilon^{-2n} \, \mathbf{I}_{j_0}+ \epsilon^{-2n} \, |\left(D u\right)_{B_{j_0}}|\\
& \quad + \frac{8}{3} \, C_1 \, \epsilon^{-2n-1}  \int_{0}^{2R_{1}}\left[\left(\tb_{B_{\rho}(x_0)}|f-\beta|^{p'}dx\right)^{1/p} + \rho^{1/(p-1)}\left(\tb_{B_{\rho}(x_0)}|g|^{\frac{np}{np-n+p}}dx\right)^{\frac{np-n+p}{(np-n)p}}  \right]\frac{d\rho}{\rho}\\
&\quad + \frac{8}{3} \, C_2 \, \epsilon^{-2n-1} \, |D u(x_0)|^{2-p}  \int_{0}^{2 R_{1}}\left[\left(\tb_{B_{\rho}(x_0)}|f-\beta|^{p'}dx\right)^{1/p'} + \rho\left(\tb_{B_{\rho}(x_0)}|g|^{\frac{np}{np-n+p}}dx\right)^{\frac{np-n+p}{np}}  \right] \, \frac{d\rho}{\rho} 
\\
& \quad + \frac{|D u(x_0)|}{10}.
\end{align*}
It remains to estimate the third term on the right-hand side of the above inequality.
Since $C_2 = 0$ when $p \in [2,n)$, we need only focus on $p \in (1,2)$.
In this case, it follows from Young's inequality that
\begin{align*}
|D u(x_0)|
&  \le \frac{8}{3} \, \epsilon^{-2n} \, \mathbf{I}_{j_0} + \epsilon^{-2n} \, |\left(D u\right)_{B_{j_0}}|+	\frac{1}{5}|D u(x_0)|\\
& \quad + C(n,p,\Lambda,\epsilon) \, \int_{0}^{2R_{1}}\left[\left(\tb_{B_{\rho}(x_0)}|f-\beta|^{p'}dx\right)^{1/p} + \rho^{1/(p-1)}\left(\tb_{B_{\rho}(x_0)}|g|^{\frac{np}{np-n+p}}dx\right)^{\frac{np-n+p}{(np-n)p}}  \right] \, \frac{d\rho}{\rho}.
\end{align*}

Either way we always have
\begin{align*}
|D u(x_0)|
&  \le 40 \epsilon^{-2n} \, T_{j_0} \\
& \quad {} + C(n,p,\Lambda,\epsilon)\int_{0}^{2R_{1}}\left[\tb_{B_{\rho}(x_0)}|f-\beta|^{p'}dx + \rho^{p'}\left(\tb_{B_{\rho}(x_0)}|g|^{\frac{np}{np-n+p}}dx\right)^{\frac{np-n+p}{np-n}} \right] \, \frac{d\rho}{\rho}
\end{align*}
which implies \eqref{maines} as desired.

{\bf Case 3}: Suppose $
T_j\leq |D u(x_0)|$ for all $j \in \{2,3,4,\ldots\}$.
Then we deduce from \eqref{es1}, \eqref{es0} and \eqref{es00} that 
\begin{align*}\nonumber
&|\left(D u\right)_{B_{k+1}}|\\
& \leq  \frac{4}{3} \, \epsilon^{-n} \, \mathbf{I}_{j_0}+|\left(D u\right)_{B_{j_0}}|\\
& \quad + \frac{4}{3} \, C_1 \, \epsilon^{-n-1}  \int_{0}^{2R_{1}}\left[\tb_{B_{\rho}(x_0)}|f-\beta|^{p'}dx + \rho^{p'}\left(\tb_{B_{\rho}(x_0)}|g|^{\frac{np}{np-n+p}}dx\right)^{\frac{np-n+p}{np-n}} \right]\frac{d\rho}{\rho}\\
&\quad + \frac{4}{3} \, C_2 \, \epsilon^{-n-1}  \, |D u(x_0)|^{2-p}  \int_{0}^{2 R_{1}}\left[\left(\tb_{B_{\rho}(x_0)}|f-\beta|^{p'}dx\right)^{p-1} + \rho^p\left(\tb_{B_{\rho}(x_0)}|g|^{\frac{np}{np-n+p}}dx\right)^{\frac{np-n+p}{n}} \right] \, \frac{d\rho}{\rho}
\\
& \quad + \frac{|D u(x_0)|}{10}
\end{align*}
for all $k \in \{2,3,4,\ldots\}$.

Simplifying the above estimate further and then letting $k\longrightarrow\infty$ we arrive at
\begin{align*}
& |D u(x_0)|\\
& \leq  \frac{4}{3} \, \epsilon^{-n}  \, \mathbf{I}_{j_0}+|\left(D u\right)_{B_{j_0}}|\\
& \quad + \frac{4}{3} \, C_1 \, \epsilon^{-n-1}  \int_{0}^{2R_{1}}\left[\tb_{B_{\rho}(x_0)}|f-\beta|^{p'}dx + \rho^{p'}\left(\tb_{B_{\rho}(x_0)}|g|^{\frac{np}{np-n+p}}dx\right)^{\frac{np-n+p}{np-n}} \right] \, \frac{d\rho}{\rho}\\
&\quad + \frac{4}{3} \, C_2 \, \epsilon^{-n-1} \, |D u(x_0)|^{2-p}  \int_{0}^{2 R_{1}}\left[\left(\tb_{B_{\rho}(x_0)}|f-\beta|^{p'}dx\right)^{p-1} + \rho^p\left(\tb_{B_{\rho}(x_0)}|g|^{\frac{np}{np-n+p}}dx\right)^{\frac{np-n+p}{n}} \right] \, \frac{d\rho}{\rho}\\
& \leq  4 \, \epsilon^{-n}  \, T_{j_0}\\
& \quad + \frac{4}{3} \, C_1 \, \epsilon^{-n-1}  \int_{0}^{2R_{1}}\left[\tb_{B_{\rho}(x_0)}|f-\beta|^{p'}dx + \rho^{p'}\left(\tb_{B_{\rho}(x_0)}|g|^{\frac{np}{np-n+p}}dx\right)^{\frac{np-n+p}{np-n}} \right] \, \frac{d\rho}{\rho}\\
&\quad + \frac{4}{3} \, C_2 \, \epsilon^{-n-1} \, |D u(x_0)|^{2-p}  \int_{0}^{2 R_{1}}\left[\left(\tb_{B_{\rho}(x_0)}|f-\beta|^{p'}dx\right)^{p-1} + \rho^p\left(\tb_{B_{\rho}(x_0)}|g|^{\frac{np}{np-n+p}}dx\right)^{\frac{np-n+p}{n}} \right] \, \frac{d\rho}{\rho}.
\end{align*}
Now \eqref{maines} follows after an application of Young's inequality as we did in the last part of Case 2. 

Thus the proof is complete.
\end{pot}

\section{Global pointwise gradient estimates}\label{global}

This section considers a boundary counterpart of Theorem \ref{grad1}.
The main task is to prove Theorem \ref{boundary}.

Before proceeding further with details, let us simplify the procedure of proving Theorem \ref{boundary} as follows.
By using a standard approximating procedure we may assume that $u\in C^1_0(\Omega)$ is a solution of \eqref{main}. 
Let $x_0 \in \Omega$.
Then Theorem \ref{grad1} states that there exists a $C = C(n,p,\Lambda,Z) > 0$ satisfying
\begin{align}\label{Wwithtail}
|D u(x_0)| &\leq C\, {\bf F}^{2\rm diam(\Omega)}_p(f,g)(x) + C\, \Big(\tb_{B_{d(x_0)}(x_0)}|D u(y)|^{p} dy\Big)^{\frac{1}{p}}.
\end{align}
Also we infer from Lemma \ref{111120149} that
\begin{equation}\label{int0}
\begin{aligned}
\tb_{\Om}|D u|^pdx
\leq C\left[\inf_{m \in \R^n} \tb_{\Omega}|f-m|^{p'}dx + \Big[{\rm diam(\Omega)}\Big]^{p'}\left(\tb_{\Omega}|g|^{\frac{np}{np-n+p}}dx\right)^{\frac{np-n+p}{np-n}} \right]
\end{aligned}
\end{equation}
with $C = C(n,p,\Lambda)$ > 0.

These two estimates together allow us to assume that $d(x_0)\leq r_1/2$ for some sufficiently small $r_1 \in (0,1)$. 
Recall that $\Omega$ is a $(\delta,R_0)$-Reifenberg flat domain for some $R_0>0$. 
Therefore we may assume further that 
\[d(x_0)
\le \frac{r_1}{2}
\le \frac{R_0}{100} 
\le \frac{{\rm diam}(\Omega)}{1000}.
\]  
This condition is implicitly understood for the rest of the paper.

\subsection{Estimates near the boundary}

The following notation is fixed in this whole section.
Let $x_1\in \partial\Omega$ be such that $|x_1-x_0|=d(x_0)$. 
Let $r\in (0, r_1]$.

Let $w\in W_{0}^{1,p}(\Omega_{2r}(x_1))+u$ be the unique solution 
to 
\[
\left\{ \begin{array}{rcl}
- \operatorname{div}\left( {\A(x,D w)} \right) &=& 0 \quad {\rm in}~~ \Omega_{2 r}(x_1), \\ 
w &=& u \quad {\rm on}~~\partial \Omega_{2r}(x_1), 
\end{array} \right.
\]
where we write  $\Omega_r(x_1)=\Omega\cap B_r(x_1)$.

Next let $v\in w+ W_0^{1,p}(\Omega_{r}(x_1))$ be the unique solution to
\begin{equation*}
\left\{ \begin{array}{rcl}
- \operatorname{div}\left( {\A(x_1,D v)} \right) &=& 0 \quad {\rm in} ~~\Omega_{r}(x_1), \\ 
v &=& w\quad  {\rm on}~~\partial \Omega_{r}(x_1). 
\end{array} \right.
\end{equation*}

In what follows, we always tacitly extend $u$ by zero to $\mathbb{R}^n\setminus \Omega$, whence in turn extend $w$ by $u$ to
$\mathbb{R}^n\setminus \Omega_{2r}(x_1)$ and $v$ by $w$ to
$\mathbb{R}^n\setminus \Omega_{r}(x_1)$.

In this subsection we collects several technical lemmas involving $w$ and $v$.

A slight modification in the proof of \eqref{es23} (to adapt the boundary case) verifies that 
\begin{equation} \label{diff bd}
\tb_{B_{r}(x_1)}|D v-D w|^{p} dx \leq C(n,p,\Lambda) \, \omega(r)^{\tau_0} \tb_{B_{r}(x_1)}|D w|^{p} dx.
\end{equation}

A boundary counterpart of \eqref{es29} is more involved and so we include a proof.
To this end, the following well-known property is essential, cf.\ \cite[Theorem 3]{W03}.

\begin{lem}
	\label{lem:mainlem}
	Let $0<\epsilon<1$ and  $B_R$ be a ball of radius $R$ in $\mathbb{R}^n$.  
	Let $E\subset F\subset B_R$ be two measurable sets with the following properties:
	\begin{enumerate}
		\item 	$|E|<\epsilon |B_R|$ and
		\item for all $x\in B_R$ and $\rho\in (0,R]$, if $|E\cap B_\rho(x)|\geq \epsilon |B_\rho(x)|$ then $B_\rho(x)\cap B_R\subset F$.
	\end{enumerate}
	Then there exists a constant $C=C(n)$ such that $|E|\leq C \epsilon |F|$.
\end{lem}

The boundary counterpart of \eqref{es29} is as follows. 
Observe that the estimate appears in an integral form as opposed to the pointwise estimate in \eqref{es29}.
This compensates the possible irregularity which may occur to the boundary of $\Omega$.

\begin{lem} \label{Gehring bdry}
	Let $q>p$. 		
	There exists a $\delta=\delta(q)>0$ such that if $\Omega$ is a  $(\delta,R_0)$-Reifenberg flat domain then
	\begin{equation}\label{z8}
	\left(	\tb_{B_{r/800}(x_1)}|D v|^qdx\right)^{1/q}\leq  C\left(	\tb_{B_{r}(x_1)}|D v|^{p}dx\right)^{\frac{1}{p}},
	\end{equation}
	where $C = C(n,p,\Lambda) > 0$.	
	In particular, 
	$$
	\tb_{B_{\epsilon r}(x_1)}|D v|^{p}dx\leq  C\epsilon^{-\frac{pn}{q}}	\tb_{B_{r}(x_1)}|D v|^{p}dx
	$$
	for all $\epsilon\in (0,1/800)$.
\end{lem}

\begin{proof} 
	The argument presented here is inspired by \cite[Proof of Lemma 4.2]{NP3}, also see \cite{N1}, \cite{N2},\cite{NP1},\cite{NP4}.
	
	Assume that $\Omega$ is a $(\delta,R_0)$-Reifenberg flat domain.
	We divide the proof into two steps.
	
	\noindent	
	{\bf Step 1}:  Let ${\bf M}$ be the standard Hardy-Littlewood maximal function and write $\mathbf{1}_E$ to denote the characteristic function of a set $E$. Set $\rho=r/800$ and  for  $\lambda>0$ let 
	$$E_{\lambda}=\left\{z\in \Omega: \left[{\bf M}\big(\mathbf{1}_{B_{8\rho}(x_1)}|D v|^{p}\big)(z)\right]^{1/p}>\lambda\right\}\cap B_{\rho}(x_1).$$

	Our task here is to show that for all $\epsilon>0$ there exist constants 
	\[
	\delta_1=\delta_1(n,p,\Lambda,\epsilon)\in (0,1),
	\quad 
	\delta_2=\delta_2(n,p,\Lambda,\epsilon)\in (0,1)
	\quad \mbox{and} \quad
	\Lambda_0=\Lambda_0(n,p,\Lambda)>1
	\] 
	such that if $\delta\leq \delta_1$ then 
	\begin{equation}|E_{\Lambda_0\lambda}|\leq C(n) \epsilon |E_{\lambda}|\label{lam-good}
	\end{equation}
	for all
	$$
	\lambda 
	\ge 
	T_0:=\delta_2^{-1}\left(	\tb_{B_{800\rho}(x_1)} |D v|^{p}dx\right)^{\frac{1}{p}}.
	$$ 
	
	To this end let $\epsilon > 0$.
	Note that ${\bf M} : L^1(\mathbb{R}^{n})\longrightarrow L^{1,\infty}(\mathbb{R}^{n})$ is bounded.
	Therefore 
	\begin{align}\label{5hh2310131}
	\left|E_{\Lambda_0\lambda}\right|\leq \frac{C(n)}{(\Lambda_0\lambda)^{p}}\int_{B_{8\rho}(x_1)} |D v|^{p}dx
	\leq C(n)\left(\frac{\delta_2}{\Lambda_0}\right)^{p}|B_{800\rho}(x_1)|\leq \epsilon \left|B_\rho(x_1)\right|
	\end{align}
	for all $\lambda\geq T_0$, provided that $\delta_2\leq \left(\frac{800^{-n}\epsilon}{C(n)}\right)^{1/p}\Lambda_0$.

	Next we need to verify 
	\begin{equation}\label{2nd-check}
	\left|E_{\Lambda_0\lambda}\cap B_{\rho_1}(x)\right|\geq \epsilon \left|B_{\rho_1}(x)\right| 
	\quad \Longrightarrow \quad 
	B_{\rho_1}(x)\cap B_\rho(x_1)\subset E_\lambda
	\end{equation}
	for all $x\in B_{\rho}(x_1)$, $\rho_1\in(0,\rho]$ and $\lambda\geq  T_0$, provided that $\delta$ and $\delta_2$, which depend on $n$, $p$, $\Lambda$ and $\epsilon$, are sufficiently small.
	
	Once this is done, \eqref{lam-good} will follow by using \eqref{5hh2310131}, \eqref{2nd-check} in combination with Lemma \ref{lem:mainlem} whose $E := E_{\Lambda_0\lambda}$ and $F := E_\lambda$.
	
	Now we prove \eqref{2nd-check} by contraposition.
	Take $x\in B_{\rho}(x_1)$, $\rho_1\in (0,\rho]$ and $\lambda\geq  T_0$.
	Assume that 
	$$B_{\rho_1}(x)\cap B_\rho(x_1)\cap (E_\lambda)^c\not= \emptyset.$$  
	That is, there exists an $x_2\in B_{\rho_1}(x)\cap B_\rho(x_1)$ such that 
	\begin{equation}\label{x2lambda}
	\left[{\bf M}\big(\mathbf{1}_{B_{8\rho}(x_1)}|D v|^{p}\big)(x_2)\right]^{1/p}\leq \lambda.
	\end{equation}	
	
	We aim to obtain
	\begin{equation}\label{5hh2310133}
	\left|E_{\Lambda_0\lambda}\cap B_{\rho_1}(x)\right|<  \epsilon \left|B_{\rho_1}(x)\right|. 
	\end{equation}
	
	Clearly,
	$$
	\left[{\bf M}\big(\mathbf{1}_{B_{8\rho}(x_1)}|D v|^{p}\big)(y)\right]^{1/p}\leq \max\left\{\left[{\bf M}\big(\mathbf{1}_{B_{2\rho_1}(x)}|D v|^{p}\big)(y)\right]^{\frac{1}{p}},3^{\frac{n}{p}}\lambda\right\}$$
	for all $y\in B_{\rho_1}(x)$.
	
	It follows that
	\[
	E_{\Lambda_0\lambda}\cap B_{\rho_1}(x)\subset\left\{z\in \Omega: \left[{\bf M}\big(\mathbf{1}_{ B_{2\rho_1}(x)}|D v|^{p}\big)(z)\right]^{\frac{1}{p}}>\Lambda_0\lambda\right\}\cap B_{\rho}(x_1) \cap B_{\rho_1}(x)
	\]
	for all $\lambda\geq T_0$ and $\Lambda_0\geq 3^{\frac{n}{p}}$.
	
	We consider two cases.
	
	\vspace{3mm}
	{\bf Case 1}: Suppose $B_{4\rho_1}(x)\Subset \Omega$.
	
	Since
	$$
	\operatorname{div}\left( {\A(x_1,D v)} \right) = 0
	\quad \mbox{in } B_{4\rho_1}(x),
	$$
	it follows from \eqref{es29} that
	$$
	\|D v\|_{L^\infty \left(B_{3\rho_1}(x)\right)}
	\leq C_1 \left( \tb_{B_{4\rho_1}(x)} |D v|^{p}dx\right)^{\frac{1}{p}}
	\leq C_1 \left(	\tb_{B_{5\rho_1}(x_2)} |D v|^{p}dx\right)^{\frac{1}{p}},
	$$
	where $C_1 = C_1(n,p,\Lambda) > 0$.
	This estimate together with \eqref{x2lambda} give
	$$
	\|D v\|_{L^\infty(B_{3\rho_1}(x))}\le C_1\lambda.
	$$
In particular, if $\Lambda_0\geq \max\left\{3^{\frac{n}{p}}, 4 C_1\right\}$ then
	$$\|D v\|_{L^\infty(B_{3\rho_1}(x))}\leq \frac{1}{2}\Lambda_0\lambda$$ and so
	\[
	E_{\Lambda_0\lambda} = \emptyset.
	\]
	
Hence \eqref{5hh2310133} is vacuously true.
	
	\vspace{3mm}
	{\bf Case 2}: Suppose $\overline{B_{4\rho_1}(x)}\cap\Omega^{c}\not=\emptyset$. 
	
	Let $x_3\in\partial \Omega$ be such that $|x_3-x|=\text{dist}(x,\partial\Omega)$.  We have 
	$$B_{2\rho_1}(x)\subset B_{6\rho_1}(x_3)\subset B_{600\rho_1}(x_3)\subset B_{605\rho_1}(x_2).$$
	Thanks to  \cite[Theorem 2.9]{MP13},	for any $\eta>0$ there exists $\delta_1=\delta_1(n,p,\Lambda,\eta)$ be such that the following holds. If $\delta\leq \delta_1$ then there exists a function $\tilde{v}\in W^{1,\infty}(B_{6\rho_1}(x_3))$  such that 
	$$
	\|D \tilde{v}\|_{L^\infty(B_{6\rho_1}(x_3))}\leq C_2 \left(\tb_{B_{600\rho_1}(x_3)}|D v|^{p}dx\right)^{1/p},
	$$
	where $C_2 = C_2(n,p,\Lambda) > 0$ and		
	$$
	\left(\tb_{B_{6\rho_1}(x_3)}|D (v-\tilde{v})|^{p}dx\right)^{\frac{1}{p}}\leq  \eta\left(\tb_{B_{600\rho_1}(x_3)}|D v|^{p}dx\right)^{1/p}.
	$$
	
	Note that if $\rho_1\leq \rho/100$, then 
	$$\left(\tb_{B_{600\rho_1}(x_3)}|D v|^{p}dx\right)^{1/p}\leq 2^{\frac{n}{p}}\left[{\bf M}\big(\mathbf{1}_{B_{8\rho}(x_1)}|D v|^{p}\big)(x_2)\right]^{1/p}\leq 2^{\frac{n}{p}+1}\lambda,$$
	and if $\rho_1\geq  \rho/100$, then since $\rho_1\leq \rho$,
	$$\left(\tb_{B_{600\rho_1}(x_3)}|D v|^{p}dx\right)^{1/p}\leq10^{\frac{3n}{p}}\left(\tb_{B_{800\rho}(x_1)}|D v|^{p}dx\right)^{1/p}\leq 10^{\frac{3n}{p}} \delta_2\lambda. $$
	
	Hence, 
	$$
	\|D \tilde{v}\|_{L^\infty\left(B_{2\rho_1}(x)\right)}\leq 10^{\frac{3n}{p}}C_2 \lambda,$$
	and
	$$\left(\tb_{B_{2\rho_1}(x)}|D (v-\tilde{v})|^{p}dx\right)^{\frac{1}{p}}\leq 10^{\frac{4n}{p}}\eta \lambda.$$
	
	Choosing $\Lambda_0= \max\left\{3^{\frac{n}{p}},4 C_1, 2^{\frac{1}{p}} 10^{\frac{3n}{p}} C_2\right\}$, we derive 
	\begin{align*}
	|E_{\Lambda_0\lambda}\cap B_{\rho_1}(x)|&\leq \left|\left\{\left[{\bf M}\big(\mathbf{1}_{ B_{2\rho_1}(x)}|D (v-\tilde{v})|^{p}\big)\right]^{\frac{1}{p}}>2^{-\frac{1}{p}}\Lambda_0\lambda\right\}\right|\\&\leq  \frac{C(n)}{\left(2^{-\frac{1}{p}}\Lambda_0\lambda\right)^{p}} \int_{B_{2\rho_1}(x)}|D (v-\tilde{v})|^{p}
	\\&\leq  \frac{2 C(n)}{\left(\Lambda_0\lambda\right)^{p}} \left(10^{\frac{4n}{p}}\eta \lambda\right)^{p}|B_{2\rho_1}(x)|\\&< \epsilon \left|B_{\rho_1}(x)\right|,
	\end{align*}
	for 
	$$
	\eta = \left(\frac{\epsilon}{10^{5n}C(n)}\right)^{1/p}.
	$$ 
	This gives \eqref{5hh2310133}.

	\noindent
	{\bf Step 2}: We will prove \eqref{z8}.
	
	Set $\lambda_0=\Lambda_0 T_0$.
	Thanks to  \eqref{lam-good},
	\begin{equation} \label{close}
	\begin{aligned} 
	&\int_{B_{\rho}(x_1)}	\left[{\bf M}\big(\mathbf{1}_{B_{8\rho}(x_1)}|D v|^{p}\big)\right]^{q/p} dx
	\\
	&= q\int_{0}^{\infty} \lambda^{q-1} \, \left|\left\{ \left[{\bf M}\big(\mathbf{1}_{B_{8\rho}(x_1)}|D v|^{p}\big)\right]^{1/p}>\lambda\right\}\cap B_{\rho}(x_1)\right| \, d\lambda
	\\ 
	& \leq  q\int_{0}^{\lambda_0} \lambda^{q-1} \, |B_{\rho}(x_1)| \, d\lambda
	\\
	& \quad + 
	C(n) \, q\epsilon\int_{\lambda_0}^{\infty}\lambda^{q-1} \, \left|\left\{\left[{\bf M}\big(\mathbf{1}_{B_{8\rho}(x_1)}|D v|^{p}\big)\right]^{1/p}>\lambda/\Lambda_0\right\}\cap B_{\rho}(x_1)\right| \, d\lambda
	\\
	&\leq  \lambda_0^{q} \, |B_{\rho}(x_1)| 
	+ C(n) \, \Lambda_0^q\epsilon \int_{B_{\rho}(x_1)}	\left[{\bf M}\big(\mathbf{1}_{B_{8\rho}(x_1)}|D v|^{p}\big)\right]^{q/p} dx.
	\end{aligned}
	\end{equation}
	A limiting argument then justifies that 
	$$
	\int_{B_{\rho}(x_1)} \left[{\bf M}\big(\mathbf{1}_{B_{8\rho}(x_1)}|D v|^{p}\big)\right]^{q/p} dx < \infty.
	$$ 
	With this in mind, we choose $\epsilon=\frac{1}{2 C(n)\Lambda_0^q}$ in \eqref{close} to obtain
	\begin{equation*}
	\tb_{B_{\rho}(x_1)}|D v|^{q}dx
	\le \lambda_0^q
	= \Lambda_0^q \, T_0^q 
	=  \Lambda_0^q \, \delta_2^{-q} \, \left(	\tb_{B_{800\rho}(x_1)} |D v|^{p}dx\right)^{\frac{q}{p}}
	=  \Lambda_0^q \, \delta_2^{-q} \, \left(	\tb_{B_{r}(x_1)} |D v|^{p}dx\right)^{\frac{q}{p}},
	\end{equation*}
	where we used $\rho=r/800$ in the last step.
	
	Thus \eqref{z8} follows with $C = C(n,p,\Lambda) = \Lambda_0 \, \delta_2^{-1}$.
\end{proof}

A comparison estimate up to boundary is stated below.
This provides an analogue of Lemma \ref{111120149}.
The proof of this statement is similar to that of Lemma \ref{111120149} and hence is omitted.

\begin{lem}\label{111120149+} 
Let $u \in W_{0}^{1,p}(\Omega)$ be a weak solution of \eqref{main} and let  $w$ be as in \eqref{thuannhat1}. Then there are constants $C = C(n,p) > 0$ and $C' = C'(n,p) > 0$ such that
	\begin{align*}
	& \tb_{B_{2r}(x_1)}|Du-Dw|^pdx\\
	& \leq C \, \left[\inf_{m \in \R^n} \tb_{B_{2r}(x_1)}|f-m|^{p'}dx + r^{p'}\left(\tb_{B_{2r}(x_1)}|g|^{\frac{np}{np-n+p}}dx\right)^{\frac{np-n+p}{np-n}} \right]\\
	& \,\, +  C' \, \left[\inf_{m \in \R^n} \left(\tb_{B_{2r}(x_1)}|f-m|^{p'}dx\right)^{p-1} + r^p\left(\tb_{B_{2r}(x_1)}|g|^{\frac{np}{np-n+p}}dx\right)^{\frac{np-n+p}{n}} \right]  \left(\tb_{B_{2r}(x_1)}|D u|^{p}dx\right)^{2-p}.
	\end{align*}
\end{lem}

\begin{rem} \label{rem2}
As in Remark \ref{rem1} we can choose $C' = 0$ when $p \in [2,n)$.
\end{rem}

\subsection{Proof of Theorem \ref{boundary}}

We are now in a position to tackle Theorem \ref{boundary}. 
The following technical lemma, which can be found in \cite[Lemma 3.13]{ACM} (also cf.\ \cite[Lemma 3.4]{HL}), will be in need. 
\begin{lem}\label{lem00}
	Let $\phi$ be a nonnegative and nondecreasing functions on $(0,Z]$. 
	Let $A$, $B$, $\delta$, $\beta$ be nonnegative constants with $\delta>\beta$ and $\eta \in (0,1)$.
	Suppose that
	\begin{align*}
	\phi(\rho)\leq A\left[\left(\f{\rho}{R}\right)^\delta+\eta\right]\phi(R)+BR^\beta
	\end{align*} for all $0<\rho\leq R\leq Z$.
	Let $\gamma\in [\beta,\delta)$.
	If $\eta <(2A)^{\delta/(\gamma-\delta)}$ then there exists a $C=C(\delta,\beta, \gamma,A) > 0$ such that
	\begin{align*}
	\phi(\rho)\leq C \left(\f{\rho}{R}\right)^{\gamma}\phi(R)+ C B \rho^\beta
	\end{align*}
	for all $0<\rho\leq R\leq Z$.
\end{lem}
Now we prove Theorem \ref{boundary}.

\begin{pot}{boundary}
Let $r\leq r_1$ and denote $B_r=B_r(x_1)$.
Let $\sigma \in (0,n)$.
By Lemma \ref{Gehring bdry}, there exists a constant $C_0 = C_0(n,p,\Lambda) \ge 1$ such that
\begin{align*}
\int_{B_{\epsilon r}}|D v|^{p} \, dz\leq C_0	\epsilon^{\sigma}\int_{B_{r}}|D v|^{p} \, dz 
\end{align*}
for all $\epsilon\in (0,1/800)$. 
Therefore 
\begin{align}\nonumber
\int_{B_{\epsilon r}} |D u|^{p} \, dz
&\leq 4^p \left( \int_{B_{\epsilon r}}|D v|^{p} \, dz+ \int_{B_{\epsilon r}}|D v-D w|^{p} \, dz
+ \int_{B_{\epsilon r}}|Dw-Du|^{p} \, dz \right)
\nonumber
\\ 
& \nonumber\leq 4^p \left( C_0 \epsilon^{\sigma}\int_{B_{ r}}|Dv|^{p} \, dz
+ \int_{B_{r}}|Dv-D w|^{p} \, dz
+ \int_{B_{r}}|Dw-Du|^{p} \, dz \right)
\\ & \leq 2 C_0 4^p \left( \epsilon^{\sigma} \int_{B_{r}}|D u|^{p} \, dz
+ \int_{B_{r}}|Dv-D w|^{p} \, dz
+ \int_{B_{r}}|Dw-Du|^{p} \, dz \right)
\nonumber
\\ 
& \leq 2 C_0 4^p \left( \big[ \epsilon^{\sigma}+\omega(r)^{\tau_0 p} \big] \int_{B_{r}}|D u|^{p} \, dz
+ \int_{B_{r}}|D w|^{p} \, dz
+ \int_{B_{\epsilon r}}|Dw-Du|^{p} \, dz \right)
\nonumber
\\
&\le 2 C_0 4^p \left( \big[ \epsilon^{\sigma}+\omega(r)^{\tau_0 p} \big] \int_{B_{ r}}|D u|^{p} \, dz
+ \int_{B_{r}}|Dw-Du|^{p} \, dz \right)
\label{es10}
\end{align}
for all $\epsilon\in (0,1/800)$, where we used the triangle inequality in the first, third and fifth steps as well as \eqref{diff bd} in the fourth step.

Next we will use Lemma \ref{111120149+} to bound the second term on the right-hand side of \eqref{es10}.
Accordingly there exist constants $C_1 = C_1(n,p) > 0$ and $C_2 = C_2(n,p) \ge 0$ such that 
\begin{align*}
&\int_{B_{\epsilon r}}|Dw-Du|^{p} \, dz
\\
& \leq C_1 \left[\tp_{B_{r}}|f-m|^{p'} \, dz + \left(\tp_{B_{r}}|g|^{\frac{np}{np-n+p}} \, dz\right)^{\frac{np-n+p}{np-n}} \right]
\\
& \quad  +  C_2 \left[\left(\tp_{B_{r}}|f-m|^{p'} \, dz\right)^{p-1} + \left(\tp_{B_{r}}|g|^{\frac{np}{np-n+p}} \, dz\right)^{\frac{np-n+p}{n}} \right] \left(\tp_{B_{r}}|D u|^{p} \, dz\right)^{2-p}\\
& \leq \eta \int_{B_{ r}}|D u|^{p} \, dz
\\
& \quad +  \left( C_1+(p-1)C_2 \left(\frac{\eta}{(C_2+1)(2-p)}\right)^{\frac{p-2}{p-1}} \right)  \left[\tp_{B_{r}}|f-m|^{p'} \, dz + \left(\tp_{B_{r}}|g|^{\frac{np}{np-n+p}} \, dz\right)^{\frac{np-n+p}{np-n}} \right]
\end{align*}
for all $\eta > 0$ and $m \in \R^n$, where we used Young's inequality in the last step. 
As mentioned in Remark \ref{rem2} we may take $C_2 = 0$ when $p \in [2,n)$.

Consequently we obtain
\begin{align*}
\int_{B_{\epsilon r}} |D u|^{p} \, dz
&\le 2 C_0 4^p \big[ \epsilon^{\sigma}+\omega(r)^{\tau_0 p} + \eta \big]\int_{B_{ r}}|D u|^{p} \, dz
\\
& \quad + 2 C_0 4^p \left( C_1+(p-1)C_2 \left(\frac{\eta}{(C_2+1)(2-p)}\right)^{\frac{p-2}{p-1}} \right)  
\\
& \quad \quad \times
\left[\tp_{B_{r}}|f-m|^{p'} \, dz + \left(\tp_{B_{r}}|g|^{\frac{np}{np-n+p}} \, dz\right)^{\frac{np-n+p}{np-n}} \right]
\end{align*}
for all $\epsilon \in (0,1/800)$, $\eta > 0$ and $m \in \R^n$.

By appropriately scaling the constants, we may choose $C = C(n,p,\Lambda) > 0$ and $C_\eta = C_\eta(n,p,\Lambda,\eta) > 0$ such that the above estimate also holds for all $\epsilon\in (0,2)$.
Consequently,
\begin{align*}
\int_{B_{\rho}(x_1)}|D u|^{p} \, dz
& \leq C	\left[ \left(\frac{\rho}{R}\right)^{\sigma} + \omega(R)^{\tau_0 p} +\eta \right] \, \int_{B_{R}(x_1)}|D u|^{p} \, dz\\
&\quad + C_\eta \, R^{\sigma}r_1^{n-\sigma}\int_{0}^{2 {\rm diam(\Omega)}}\left[\tb_{B_{\rho}(x)}|f-m|^{p'} \, dz + \rho^{p'}\left(\tb_{B_{\rho}(x)}|g|^{\frac{np}{np-n+p}} \, dz\right)^{\frac{np-n+p}{np-n}} \right]\frac{d\rho}{\rho},
\end{align*}
for all $0<\rho\leq R\leq 2r_1$, $\eta > 0$ and $m \in \R^n$.

Now Lemma \ref{lem00} with
$$
\phi(r) :=\int_{B_{r}(x_1)}|D u|^{p} \, dz
\quad \mbox{with } r\in(0,2r_1),
$$ 
gives
\begin{align*}
\int_{B_{\rho}(x_1)}|D u|^{p} \, dz
& \leq C \, \left(\frac{\rho}{R}\right)^{\sigma}\int_{B_{R}(x_1)}|D u|^{p} \, dz\\
&\quad +  C \, \rho^{\sigma}r_1^{n-\sigma} \int_{0}^{2 {\rm diam(\Omega)}}\left[\tb_{B_{\rho}(x)}|f-m|^{p'} \, dz + \rho^{p'}\left(\tb_{B_{\rho}(x)}|g|^{\frac{np}{np-n+p}} \, dz\right)^{\frac{np-n+p}{np-n}} \right]\frac{d\rho}{\rho},
\end{align*}
provided that $r_1$ and $\eta$ are sufficiently small. 
Hereafter $C = C(n,p,\Lambda) > 0$.

Next let $x_0 \in \Omega$.
A specific choice of $R=2r_1$ and $\rho=2 d(x_0)$ in the last display leads to 
\begin{align*}
&\tb_{B_{2d(x_0)}(x_1)}  |D u|^{p} \, dz\\
& \leq C\left(\frac{r_1}{d(x_0)}\right)^{n-\sigma}
\\
& \quad \times \left(\tb_{B_{ 2 r_1}(x_1)}|D u|^{p} \, dz +  \int_{0}^{2 {\rm diam(\Omega)}}\Bigg[\tb_{B_{\rho}(x)}|f-m|^{p'} \, dz + \rho^{p'}\Bigg(\tb_{B_{\rho}(x)}|g|^{\frac{np}{np-n+p}} \, dz\Bigg)^{\frac{np-n+p}{np-n}} \Bigg]\frac{d\rho}{\rho} \right).
\end{align*}

This in turn implies 
\begin{align*}
&\tb_{B_{d(x_0)}(x_0)}  |D u|^{p} \, dz\\
& \leq C\left(\frac{r_1}{d(x_0)}\right)^{n-\sigma}
\\
& \quad \times \left(\tb_{B_{ 2 r_1}(x_1)}|D u|^{p} \, dz +  \int_{0}^{2 {\rm diam(\Omega)}}\Bigg[\tb_{B_{\rho}(x)}|f-m|^{p'} \, dz + \rho^{p'}\Bigg(\tb_{B_{\rho}(x)}|g|^{\frac{np}{np-n+p}} \, dz\Bigg)^{\frac{np-n+p}{np-n}} \Bigg]\frac{d\rho}{\rho} \right)\\
& \leq \frac{C}{d(x_0)^{n-\sigma}} \int_{0}^{2 {\rm diam(\Omega)}}\left[\tb_{B_{\rho}(x)}|f-m|^{p'} \, dz + \rho^{p'}\left(\tb_{B_{\rho}(x)}|g|^{\frac{np}{np-n+p}} \, dz\right)^{\frac{np-n+p}{np-n}} \right]\frac{d\rho}{\rho},
\end{align*}
where we used \eqref{int0} and the fact that $r_1 < 1$ in the second step.

Now \eqref{ine3} follows directly from this last estimate and \eqref{Wwithtail}, which completes our proof.
\end{pot} 


\end{document}